\documentclass[reqno]{amsart}

\usepackage{amsmath,amssymb,graphicx
}
\usepackage{verbatim} 

\usepackage[latin1]{inputenc}
\usepackage{enumerate}
\usepackage{hyperref}
\usepackage{amsthm}

\graphicspath{{Figs/}}
\newtheorem{theorem}{Theorem}[section]
\newtheorem*{theorem*}{Theorem}
\newtheorem{proposition}[theorem]{Proposition}
\newtheorem{corollary}[theorem]{Corollary}
\newtheorem{lemma}[theorem]{Lemma}
\theoremstyle{definition}
\newtheorem{remark}[theorem]{Remark}
\newtheorem{definition}[theorem]{Definition}

\newcommand{\Gen}{G}

\newcommand{\Cylset}{\mathcal{O}}
\newcommand{\Cyl}{O}

\newcommand\touch[1]{\check{{#1}}}

\renewcommand{\H}{\operatorname{Heap}}

\newcommand{\Aaff}{\aff{A}}

\newcommand{\aff}[1]{\widetilde{#1}}

\newcommand\qbi[3]{{{#1}\atopwithdelims[]{#2}}_{#3}}
\newcommand\bi[2]{{{#1}\atopwithdelims(){#2}}}

\allowdisplaybreaks

\begin{document}

\title{Long fully commutative elements in affine Coxeter groups}

\author[Fr\'ed\'eric Jouhet]{Fr\'ed\'eric Jouhet}
\address{Institut Camille Jordan, Universit\'e Claude Bernard Lyon 1,
69622 Villeurbanne Cedex, France}
\email{jouhet@math.univ-lyon1.fr}
\urladdr{http://math.univ-lyon1.fr/{\textasciitilde}jouhet}

\author[Philippe Nadeau]{Philippe Nadeau}
\address{CNRS, Institut Camille Jordan, Universit\'e Claude Bernard Lyon 1,
69622 Villeurbanne Cedex, France}
\email{nadeau@math.univ-lyon1.fr}
\urladdr{http://math.univ-lyon1.fr/{\textasciitilde}nadeau}

\thanks{The authors thank Fran\c{c}ois Bergeron for initiating this research while both of them were staying at LaCIM }

\date{\today}

\subjclass[2010]{}


\begin{abstract}
An element of a Coxeter group $W$ is {\em fully commutative} if any two of its reduced decompositions are related by a series of transpositions of adjacent commuting generators. In the preprint \emph{Fully commutative elements  in finite and affine Coxeter groups}, Biagioli, Jouhet and Nadeau proved among other things that, for each irreducible affine Coxeter group, the sequence counting fully commutative elements with respect to length is ultimately periodic. In the present work, we study this sequence in its periodic range for each of these groups, and in particular we determine the minimal period. We also observe that in type $\Aaff$ we get an instance of the cyclic sieving phenomenon. 
\end{abstract}

\maketitle


\section*{Introduction}
\label{sec:intro}

Let $W$ be a Coxeter group. An element $w \in W$ is said to be {\em fully commutative} if any reduced expression for $w$ can be obtained from any other one by transposing adjacent pairs of commuting generators. Fully commutative elements were extensively studied by Stembridge in a series of papers~\cite{St1,St2,St3} where, among others, he classified the Coxeter groups having a finite number of fully commutative elements and enumerated them in each case. It is known that fully commutative elements in Coxeter groups index a basis for a quotient of the associated (generalized) Temperley--Lieb algebra (\cite{Fan,Graham}).

If $W^{FC}$ denotes the subset of fully commutative (FC) elements of $W$, let $W^{FC}_\ell$ be the number of FC elements of Coxeter length $\ell$.
 In the case of the affine symmetric group, Hanusa and Jones~\cite{HanJon} proved that the corresponding counting sequence (or \emph{growth function}) $(W^{FC}_\ell)_{\ell\geq 0}$ is ultimately periodic. In \cite{BJN}, Biagioli and the two authors generalized these results to all finite or affine Coxeter groups, by using the theory of \emph{heaps} and encoding fully commutative elements by various classes of lattice walks. One of the results is the following.
\begin{theorem*}[\cite{BJN}]
For each irreducible, classical affine Coxeter group $W$, the growth function $(W^{FC}_\ell)_{\ell\geq 0}$ is ultimately periodic with following period:
  \[\begin{array}{ l || c|c|c|c}
    \textsc{Affine Type} &\aff{A}_{n-1}&\aff{C}_n&\aff{B}_{n+1}&\aff{D}_{n+2}\\ \hline
    \textsc{Periodicity} &n&n+1&(n+1)(2n+1)&n+1 \\
   \end{array}\]
\end{theorem*}

In fact the full generating functions $W^{FC}(q):=\sum_{w\in W^{FC}}q^{\ell(w)}$, for $W$ affine or finite, were computed in~\cite{BJN}, as was the precise start of periodicity.
Similar results are proved in~\cite{BJN_families} for the subset of $W^{FC}$ of involutions. In particular, the corresponding growth functions turn out to be also ultimately periodic with the following periods: 
  \[\begin{array}{ l || c|c|c|c}
    \textsc{Affine Type} &\aff{A}_{2n-1}&\aff{C}_n&\aff{B}_{n+1}&\aff{D}_{n+2} \\ \hline
    \textsc{Periodicity} &2n&2(n+1)&2(n+1)(2n+1)&2(n+1)\\
   \end{array}\]

In view of these results, a natural question arises regarding the minimal periods of all these growth functions. In the present paper, we will determine them for all the classical affine types. To this aim, we will use the encoding of fully commutative elements by heaps, and derive from their classifications proved in~\cite{BJN} new expressions for $W^{FC}(q)$ (up to a polynomial). The latter involve generating functions for integer partitions and yield, through arithmetical investigations, the desired minimal periods.

More precisely, we will show how heaps associated with FC elements corresponding to the periodic part of $W^{FC}(q)$, which we call \emph{long fully commutative elements}, can be enumerated, according to the length, through families of integer partitions. In type $\Aaff$, these considerations exhibit a \emph{cyclic sieving phenomenom} (see for instance~\cite{reiner2004cyclic, sagan2010cyclic}). In all classical affine types, our periodicity results  can be summarized as follows.

\begin{theorem*}[Minimal Periods] Let $n\geq 2$.
In type $\Aaff_{n-1}$, the minimal ultimate period of the growth function $(W^{FC}_\ell)_{\ell\geq 0}$ is equal to $p^{\alpha-1}$ if $n=p^\alpha$ for a prime $p$ and a positive integer $\alpha$, and to $n$ otherwise. In type $\aff{C}_{n}$ (\emph{resp.}  $\aff{B}_{n+1}$, \emph{resp.} $\aff{D}_{n+2}$), the minimal period is given by $2m+1$ (\emph{resp.} $(2m+1)(2n+1)$, \emph{resp.} $n+1$) where  $2m+1$ is the largest odd divisor of $n+1$. 
\end{theorem*}

We also determine in the same way the minimal periods for the corresponding affine involutions. Moreover, we compute expressions for the number of FC elements of a given large enough length $\ell$, for each of these types.\\

This paper is organized as follows. In Section~\ref{sec:heapsfullycomm}, we recall definitions and properties concerning heaps and fully commutative elements. In Section~\ref{sec:prelim}, we prove useful elementary results  on ultimately periodic sequences, specializations of the $q$-binomial coefficients, and  recall some classical identities on integer partitions. Periodicity results regarding growth functions for long FC elements and involutions of type $\Aaff_{n-1}$ are given in Section~\ref{Aaffine}, while the other classical affine types are treated in Section~\ref{CBDaffine}. From these results, we compute some exact and asymptotic evaluations in Section~\ref{section:num}. Finally, a manifestation of the cyclic sieving phenomenon occurring in type $\Aaff_{n-1}$ is explained in Section~\ref{sec:csp}.

\section{Heaps and Fully commutative elements}
\label{sec:heapsfullycomm}
In this section, we recall the definition of heaps and its relation with fully commutative elements in Coxeter groups. We finish by recalling relevant results from~\cite{BJN} regarding fully commutative elements in (classical) affine types.
\medskip

{\noindent \bf Heaps.} Let $\Gamma$ be a finite, simple graph with vertex set $S$.  A {\em heap} on $\Gamma$ (or $\Gamma$-heap) is a finite poset $(H,\leq)$, together with a labeling map $\epsilon:H\to\Gamma$, which satisfies the following conditions:\par
(i) For any vertex $s$ (\emph{resp.}  any edge $\{s,t\}$), the subposet $H_s:=\epsilon^{-1}(\{s\})$ (\emph{resp.} $H_{\{s,t\}}:=\epsilon^{-1}(\{s,t\})$) is totally ordered;\par 
(ii) The partial ordering $\leq$ is the smallest one containing all chains $H_s$ and $H_{\{s,t\}}$.

We write $H_{s}=\{s^{(1)}<s^{(2)}<\cdots<s^{(k)}\}$ and its elements are called {\em $s$-elements}. Two heaps on $\Gamma$ are {\em isomorphic } if there exists a poset isomorphism between them which preserves the labels. The size $|H|$ of a heap $H$ is its cardinality. Heaps were originally defined by Viennot~\cite{ViennotHeaps}; the definition we use can be found as~\cite[p.20]{GreenBook} or \cite[Definition 2.2]{KrattHeaps}.

 As defined in~\cite{BJN}, a $\Gamma$-heap $H$ is {\em alternating} if for each edge $\{s,t\}$ of $\Gamma$, the chain $H_{\{s,t\}}$ has alternating labels $s$ and $t$.
\medskip

{\noindent \bf Words and Heaps.}  Consider now words on $S$, i.e. elements of the free monoid $S^*$ generated by $S$.  Let $\sim$ be the equivalence relation on $S^*$ generated by pairs $\mathbf{u}st\mathbf{v}\sim \mathbf{u}ts\mathbf{v}$ with letters $\{s,t\}\in S$ which are \emph{not} adjacent in $\Gamma$. A {\em $\Gamma$-commutation class} is an equivalence class for this relation.

 Now given a word $\mathbf{w}=s_{1}\cdots s_{l}$ in $S^*$,  set $i\prec j$ if $i<j$ and $\{s_{i},s_{j}\}$ is an edge of $\Gamma$, and extend by transitivity to a partial ordering $\prec$ of the index set $\{1,\ldots, l\}$. This poset together with $\epsilon:i\mapsto s_{a_i}$ forms a heap whose isomorphism class we denote by $\H({\mathbf{w}})$. We have then the following fundamental result.
\begin{proposition}[Viennot \cite{ViennotHeaps}]
\label{prop:wordtoheap} The map $\mathbf{w}\mapsto \H({\mathbf{w}})$ induces a bijection between $\Gamma$-commutation classes of words and finite $\Gamma$-heaps. 
\end{proposition}
\smallskip

{\noindent \bf Full commutativity.} We refer the reader to~\cite{Humphreys} for a standard introduction to Coxeter systems.
Consider integers $m_{st}$ indexed by $S^2$ satisfying $m_{ss}=1$ and, for $s\neq t$, $m_{st}=m_{ts}\in\{2,3,\ldots\}\cup\{\infty\}$. The {\em Coxeter group} $W$ associated with $M$ is defined by generators set $S$ and relations $(st)^{m_{st}}=1$ if $m_{st}<\infty$. These relations can be rewritten as $s^2=1$ for all $s$, and $\underbrace{sts\cdots}_{m_{st}}  = \underbrace{tst\cdots}_{m_{st}},$  when $m_{st}<\infty$.

The {\em Coxeter graph} $\Gamma$ is the graph with vertex set $S$ and, for each pair $\{s,t\}$ with $m_{st}\geq 3$, an edge between $s$ and $t$ labeled by $m_{st}$; when $m_{st}=3$, the edge is usually left unlabeled since this case occurs frequently. Notice that non adjacent vertices correspond to \emph{commutation relations} $st=ts$. For $w\in W$, the {\em length} of $w$, denoted by $\ell(w)$, is the minimum $l$ of any expression $w=s_1\cdots s_l$ with $s_i\in S$. Expressions of length $\ell(w)$ are called \emph{reduced} and form the set $\mathcal{R}(w)$.
\smallskip
 
Fix a Coxeter system $(W,S)$ and let $\Gamma$ be its associated Coxeter graph.

\begin{definition} An element $w\in W$ is \emph{fully commutative} (FC) if the set $\mathcal{R}(w)$ forms a $\Gamma$-commutation class.
\end{definition}

Therefore if $w$ is a FC element and $\mathbf{w}\in \mathcal{R}(w)$, one can define $\H(w):=\H(\mathbf{w})$ and heaps of this form are called {\em FC heaps}. We have thus a bijection between FC elements and FC heaps, but one needs an intrinsic characterization of FC heaps for this to be useful. 

This was done by Stembridge in~\cite{St1}, and used in~\cite{BJN} to classify the FC heaps for all affine types. Among these FC heaps, some belong to finite families, and others to infinite families. As we are interested in ultimate periodicity in the present article, we will focus on the latter.

In the rest of this section, we therefore  recall the relevant results from~\cite{BJN} concerning long FC elements in affine types $\Aaff_{n-1}$ and $\aff{C}_n$. This is enough since long FC elements in the other classical affine types $\aff{B}_{n+1}$ and $\aff{D}_{n+2}$ are deduced from the ones in $\aff{C}_n$.
\medskip

{\bf Type $\aff{A}$}: The Coxeter graph of type $\Aaff_{n-1}$ is as follows:
\begin{minipage}[c]{3cm}
\includegraphics[height=2cm]{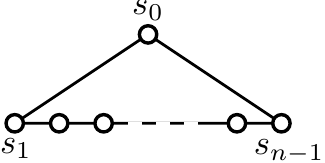}
\end{minipage}
\smallskip

FC heaps in type $\Aaff_{n-1}$ are precisely alternating heaps, as was proved in ~\cite{BJN,Gre321}. To represent them, we duplicate the set of $s_0$-elements and use one copy for the depiction of the chain $H_{\{s_0,s_{1}\}}$ and one copy for $H_{\{s_{n-1},s_0\}}$. This can be seen in Figure~\ref{fig:AtildeHeap_slanted}, left. The representation on the right is a linear deformation of the first one which can be regarded in a more visible way as ``drawn on a cylinder".

\begin{figure}[!ht]
\begin{center}
 \includegraphics[width=0.6\textwidth]{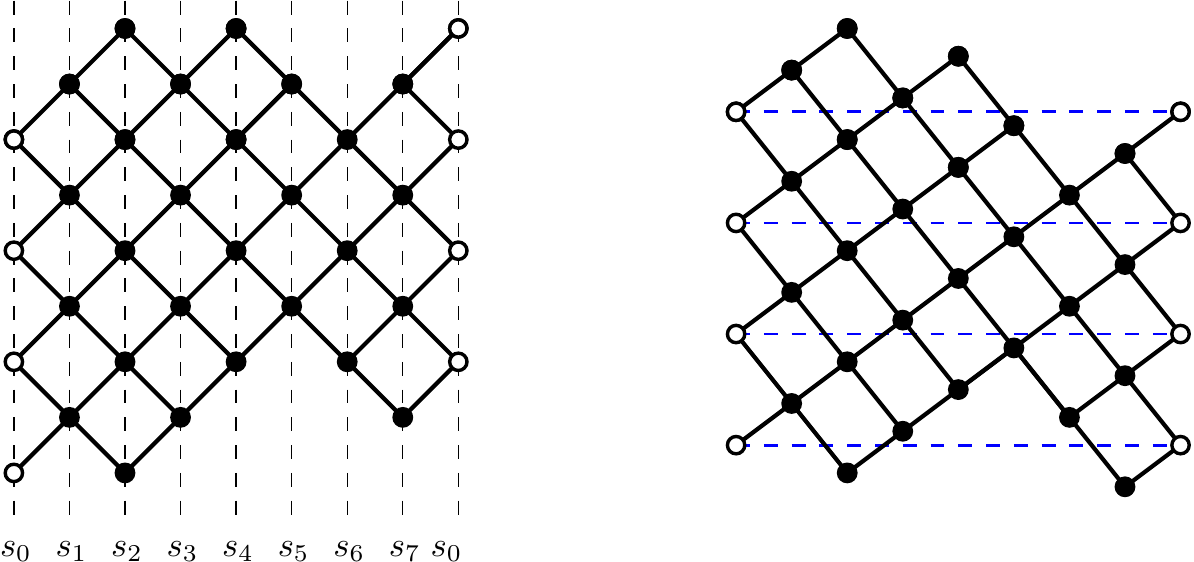}
 \caption{\label{fig:AtildeHeap_slanted} Representation of a FC heap of type $\aff{A}_7$.}
 \end{center} 
\end{figure}

 Let $\touch{\Cylset}_n$ be the set of lattice paths from $(0,i)$ to $(n,i)$ for a certain $i$, using steps $D=(1,-1),H_1=H_2=(1,0),U=(1,1)$, which stay above the $x$-axis but must touch it at some point; here $H_1,H_2$ correspond to two possible labelings for horizontal steps. The generating function $\touch{\Cyl}_n(q)$ counts such paths according to the algebraic area below them. Thanks to the work in~\cite{BJN}, the generating function for long FC elements of type $\Aaff_{n-1}$  can be expressed in terms of these walks:
\begin{equation}
 \label{eq:Aaffine}
\aff{A}_{n-1}^{FC}(q)=\frac{\touch{\Cyl}_n(q)-2}{1-q^{n}} +\mbox{a polynomial}.
\end{equation}

\medskip

{\bf Type $\aff{C}$}: the Coxeter graph of type $\aff{C}_n$ is:
\begin{minipage}[c]{3cm}
\includegraphics[height=0.8cm]{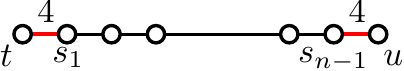}
\end{minipage}
\smallskip

It was shown in~\cite{BJN} that, apart from three finite families,  FC heaps of type $\aff{C}_n$ form two infinite families: alternating heaps and ``zig-zag'' heaps which correspond to  subwords of $(ts_1\cdots s_{n-1}us_{n-1}\cdots s_1)^\infty$ and whose generating function is easy to express. Consider  paths  from $(0,i)$ to $(n,j)$ for some nonnegative integers $i,j$, using steps $D,H_1,H_2,U$, which stay above the $x$-axis but must touch it at some point; and denote by $\touch{\Gen}_n(q)$ their generating function according to the sums of the heights of their points. Then, as shown in~\cite{BJN}, the generating function for long FC elements of type $\aff{C}_n$ can be expressed in terms of these walks:
\begin{equation}
 \label{eq:Caffine}
\aff{C}_{n}^{FC}(q)=\frac{\touch{\Gen}_n(q)}{1-q^{n+1}} + \frac{2n}{1-q}+\mbox{a polynomial}.
\end{equation}

 FC heaps of types $\aff{B}$ and $\aff{D}$ can be described based on those of types $\aff{C}$, cf. \cite{BJN}. On the level of generating functions (for long heaps), the relation is simple and will be recalled in Section~\ref{CBDaffine}.

\section{Preliminary results}
\label{sec:prelim}
In this section, we first give a general proposition on ultimately periodic sequences, and then state certain  specialization results about the $q$-binomial coefficients, which will both be useful later. Finally, we recall some classical identities on integer partitions.
\subsection{Ultimately periodic sequences}
Let $f(q)={\sum_{i\geq 0}a_iq^i}$ and $g(q)={\sum_{i\geq 0}b_iq^i}$ be two power series in $\mathbb{C}[[q]]$; we say that they are equivalent, and write $f\equiv g$, if $f-g$ is a polynomial. Equivalently, $f\equiv g$ iff the two coefficient sequences $(a_i)_{i\geq0}$ and $(b_i)_{i\geq0}$ coincide for $i$ large enough, i.e. if the set $\{i\geq0|a_i\neq b_i\}$ is finite. This is clearly an equivalence relation on power series.

Let $\mathbb{U}\subseteq \mathbb{C}$ be the group of complex roots of unity. We denote by  $\operatorname{order}(\xi)$  the multiplicative order of $\xi\in \mathbb{U}$, i.e. the smallest $m>0$ such that $\xi^m=1$.

\begin{proposition}\label{prop:partialfractionperiod} Let $\displaystyle{f(q)={\sum_{i\geq 0}a_iq^i}\in 
\mathbb{C}[[q]]}$. Then the following conditions are equivalent:
\begin{enumerate}
\item \label{it1} The sequence $(a_i)_{i\geq 0}$ is ultimately periodic.
\item \label{it2}$f(q)\equiv\frac{P(q)}{1-q^N}$ for some polynomial $P(q)$ and positive integer $N$.
\item \label{it3}$\displaystyle{f(q)\equiv\sum_{\xi\in\mathbb{U}}\frac{\alpha_{\xi}}{1-q\xi}}$, where the  $\alpha_\xi$ are complex coefficients such that $\mathbb{U}_f:=\{\xi\in\mathbb{U}; \alpha_\xi\neq 0\}$ is a finite set.
\end{enumerate}
\end{proposition}

\begin{proof}
\eqref{it1}$\Rightarrow$\eqref{it2}: By hypothesis there exist $d,N>0$ such that $a_{i+N}=a_i$ for $i\geq d$,
\begin{align*}
\text{hence } f(q)&=
\sum_{i=0}^{d-1}a_iq^i+\sum_{k\geq 0}\sum_{i=d}^{d+N-1}a_iq^{kN+i}
=\sum_{i=0}^{d-1}a_iq^i+\frac{1}{1-q^N}\left(\sum_{i=d}^{d+N-1}a_iq^{i}\right),
\end{align*}
so that $f(q)$ satisfies Condition \eqref{it2}.

\eqref{it2}$\Rightarrow$\eqref{it1}: write the euclidean division $P(q)=(1-q^N)Q(q)+R(q)$ where $\deg R<N$; one has then $f(q)=Q(q)+R(q)/(1-q^N)\equiv R(q)/(1-q^N)$ which shows that $(a_i)_{i\geq 0}$ is ultimately periodic.

\eqref{it2}$\Rightarrow$\eqref{it3}: this follows by partial fraction decomposition of $f(q)$.

\eqref{it3}$\Rightarrow$\eqref{it2}: since $\mathbb{U}_f$ is finite, there exists a positive integer $ N$ such that $\operatorname{order}(\xi)$ divides $N$ for all $\xi\in\mathbb{U}_f$. This implies that $(1-q\xi)$ divides $(1-q^N)$ in $\mathbb{C}[q]$ for all such $\xi$, and thus $f(q)$ can be written as $P(q)/(1-q^N)$.\medskip
\end{proof}

\begin{corollary}\label{coro:partialfractionperiod}
If one of the equivalent conditions of Proposition~\ref{prop:partialfractionperiod} holds,  then the minimal period in~\eqref{it1} is equal to the smallest $N$ for which~\eqref{it2} holds, and is also equal to the least common multiple of all the integers $\operatorname{order}(\xi)$, $\xi\in\mathbb{U}_f$, from condition~\eqref{it3}.
\end{corollary}
\begin{proof}
It is clear from the previous proof that the minimal period of $(a_i)_{i\geq 0}$ is equal to the smallest possible $N$ in~\eqref{it2}. It is also clear from the proof of \eqref{it3}$\Rightarrow$\eqref{it2} that the least common multiple $M$ of the numbers $\operatorname{order}(\xi)$ for $\xi\in\mathbb{U}_f$ is a valid $N$ for \eqref{it2}. Now assume for the sake of contradiction that there exists $N<M$ such that \eqref{it2} holds. Then all the poles of $P(q)/(1-q^N)$ are roots of unity with orders dividing $N$, so their least common multiple is at most $N$. But these poles form precisely the set $\mathbb{U}_f$, which is absurd.
\end{proof}

We can be more explicit about the partial fraction decomposition when $f(q)$ has the form given in \eqref{it2} above. For a positive integer $N$, set $\xi_N:=\mbox{e}^{2i\pi/N}$; then one has
\begin{equation}
\label{eq:parfrac}
\frac{P(q)}{1-q^N}\equiv \frac{1}{N}\sum_{j=0}^{n-1}\frac{P(\xi_N^{-j})}{1-q\xi_N^j}.
\end{equation}

Indeed, in the decomposition $\frac{P(q)}{1-q^N}\equiv\sum_{j=0}^{N-1}\frac{\alpha_j}{1-q\xi_N^j}$, the coefficient $\alpha_j$ is equal to
\[
\lim_{q\to \xi_N^{-j}}\frac{P(q)(1-q\xi_N^j)}{1-q^N}=P(\xi_N^{-j})\lim_{q\to \xi_N^{-j}}\frac{1-q\xi_N^j}{1-q^N}=P(\xi_N^{-j})\lim_{q\to \xi_N^{-j}}\frac{-\xi_N^j}{-Nq^{N-1}}=\frac{P(\xi_N^{-j})}{N},
\]
where we used L'H\^opital's rule in the second equality.

\subsection{$q$-binomial coefficients}
Recall that the \emph{$q$-binomial coefficients} are defined as follows
$$\qbi{n}{k}{q}:=\frac{(q;q)_n}{(q;q)_k(q;q)_{n-k}},$$
where for any complex number $a$, $(a;q)_n:=(1-a)\cdots(1-aq^{n-1})$ is the $q$-shifted factorial. These deformations of the binomials are {\em polynomials } in the variable $q$, with positive integral coefficients; as we will see in Section~\ref{sub:partitions}, these polynomials enumerate certain integer partitions. It thus makes sense to substitute any complex number for $q$, and the following specialization will be used in the sequel. 
\begin{lemma}\label{lemma:qbinrootunity}
For any nonnegative integers $n$, $k$ and $j$ satisfying $0\leq k\leq n$, we have
$$\qbi{n}{k}{\xi_n^{j}}=\left\{
    \begin{array}{cl}
        \displaystyle\bi{(n,j)}{k(n,j)/n} & \mbox{if } n\,\mbox{ divides }\, k(n,j), \\\\
        0 & \mbox{otherwise,}
    \end{array}\right.
$$ 
where $(n,j)$ denotes the greatest common divisor of $n$ and $j$.
\end{lemma}

This can be proved as a consequence of the so-called $q$-Lucas property, which has in this case a combinatorial proof (see for instance Sagan~\cite{sagancongruence}). We will give another proof, based on Stanley's~\cite[Exercise 3.45(b)]{StanleyEC1}.

\begin{proof}
 The $q$-binomial formula (see for instance \cite{GR})  can be written as:
\begin{equation}\label{eq:qbinom}
\prod_{i=0}^{n-1}(y-q^i)=\sum_{k=0}^{n}\qbi{n}{k}{q}q^{\bi{k}{2}}(-1)^ky^{n-k}.
\end{equation}
Next, setting $d:=(n,j)$ and $f:=n/d$, noticing that $\xi_n^{ij}=\xi_{n/d}^{ij/d}$ and $(n/d,j/d)=1$, we have:
$$\prod_{i=0}^{n-1}(y-\xi_n^{ij})=\prod_{t=0}^{d-1}\prod_{i=tf}^{(t+1)f-1}\left(y-\xi_{f}^{ij/d}\right)=\prod_{t=0}^{d-1}(y^{f}-1)=(y^{n/d}-1)^{d}.$$
By expanding this expression and identifying the coefficient of $y^{n-k}$ with the one in~\eqref{eq:qbinom} (where $q$ is replaced by $\xi_n^{j}$), we see that that this coefficient is $0$ unless $n$ divides $kd$. In this case, we have the identity
$$\qbi{n}{k}{\xi_n^{j}}=\bi{d}{kd/n}(-1)^{k-kd/n}\xi_n^{-j\bi{k}{2}}=(-1)^{k-kd/n}(-1)^{(k-1)kj/n}\bi{d}{kd/n}.$$
It remains to show that the exponents $k-kd/n$ and $(k-1)kj/n$ have the same parity. To see this, denote first for any integer $t$ its $2$-adic valuation by $v_2(t)$, and remark that, as $n$ divides $kd$, the number $kj/n$ is an integer. 

Now $(k-1)kj/n$ is odd if and only if both $k-1$ and $kj/n$ are odd, which is equivalent to the conditions  $v_2(k)>0$ and $v_2(k)+v_2(j)=v_2(n)$. In the same way, the integer $k-kd/n$ is odd if and only if exactly one of $k$ and $kd/n$ either $v_2(k)>0$ and $v_2(k)+v_2(d)=v_2(n)$, or $v_2(k)=0$ and $v_2(k)+v_2(d)>v_2(n)$. But, as $d$ divides $n$, this second condition is impossible. Finally, recalling that $v_2(d)=min(v_2(n), v_2(j))$, the conditions $v_2(k)>0$ and $v_2(k)+v_2(d)=v_2(n)$ are equivalent to $v_2(k)>0$ and $v_2(k)+v_2(j)=v_2(n)$.
\end{proof}

\subsection{Integer partitions}
\label{sub:partitions}
Recall that a \emph{partition} $\lambda:=(\lambda_1\geq\lambda_2\geq\cdots)$ of a nonnegative integer $n$ is a finite nonincreasing sequence of positive integers whose sum is equal to $n$, and $n=:|\lambda|$ is the \emph{size} of $\lambda$. Each of the  $\lambda_i$'s is called a {\em part} of the partition $\lambda$. A partition can be represented as a {\em Ferrers diagram}: it is a left-aligned array of boxes, such that each part $\lambda_i$ corresponds to a row of $\lambda_i$ boxes; see Figure~\ref{fig:partitions}, left, for the Ferrers diagram representing the partition $(14,10,5,5,3,2,2)$ of size $41$.

 The following is a well-known fact about $q$-binomial coefficients, and can be found for instance in~\cite{Andrews}. 
\begin{lemma}\label{lemma:partitionsboite}
For any positive integers $n$ and $k$, the generating function, according to the size, of partitions with $\lambda_1\leq n-k$ and at most $k$ parts is given by $\qbi{n}{k}{q}$.
\end{lemma}
Such partitions correspond bijectively to Ferrers diagrams which fit in a rectangle with dimensions $k\times(n-k)$. We now record the two well-known identities
\begin{equation}\label{durfee}
\sum_{k=0}^{\min(a,b)}\qbi{a}{k}{q}\qbi{b}{k}{q}q^{k^2}=\qbi{a+b}{a}{q},
\end{equation}
and
\begin{equation}\label{staircase}
\sum_{k=0}^{n}\qbi{n}{k}{q}q^{k(k+1)/2}=(-q;q)_n.
\end{equation}

We give combinatorial proofs of these; graphical illustrations are provided in Figure~\ref{fig:partitions}. The r.h.s. of~\eqref{durfee} counts Ferrers diagrams in a $a\times b$ rectangle. For such a diagram, let $k$ be the size of the Durfee square, i.e. $\lambda_k\geq k$ and is maximal with this property. Removing the square (which has size $k^2$) leaves two diagrams wich fit respectively in rectangles $(a-k)\times k$ and $k\times (b-k)$, which proves~\eqref{durfee}.

 For~\eqref{staircase}, notice the r.h.s. counts partitions with distinct parts such that $\lambda_1\leq n$. For such a partition $\lambda$, let $k$ be its number of parts and remove $k,k-1, \ldots, 1$ from the parts $\lambda_1,\lambda_2,\ldots,\lambda_k$ respectively. Discarding possible zero parts, this leaves a partition $\lambda'$ whose Ferrers diagram fits in a $k\times (n-k)$-box, which proves~\eqref{staircase}.

\begin{figure}[!ht]
\begin{center}
 \includegraphics[width=\textwidth]{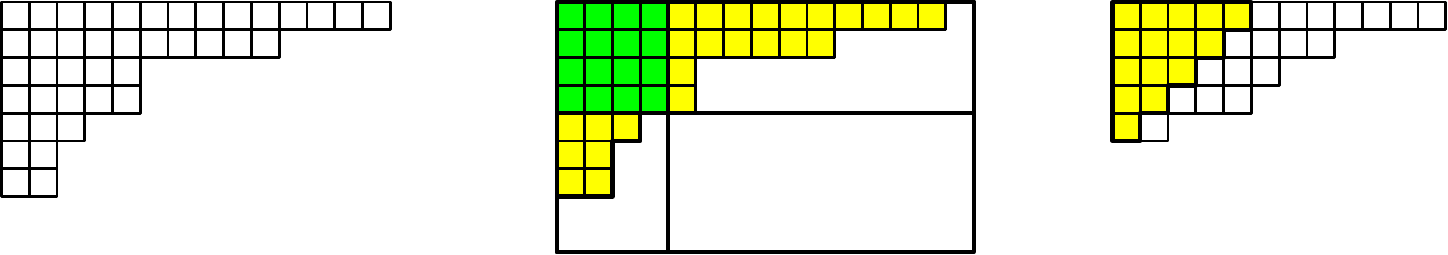}
 \caption{\label{fig:partitions} A Ferrers diagrams, and graphical illustrations of identities~\eqref{durfee} and~\eqref{staircase}.}
 \end{center} 
\end{figure}

\section[Type A affine]{Minimal period for the growth function in type $\Aaff_{n-1}$}
\label{Aaffine}

Let $a_\ell^{(n)}$ denote the number of FC elements of length $\ell$ in type $\Aaff_{n-1}$. The generating function $ \Aaff_{n-1}^{FC}(q)=\sum_{l\geq0} a_l^{(n)}q^l$ was first computed by Hanusa and Jones in~\cite{HanJon}. Up to a polynomial it can be written as:
 \begin{equation}
 \label{eq:HJ}
\aff{A}_{n-1}^{FC}(q)\equiv\frac{1}{1-q^{n}}\sum_{k=1}^{n-1}\qbi{n}{k}{q}^2.
\end{equation}
 
From Proposition~\ref{prop:partialfractionperiod} it follows that $(a_l^{(n)})_{l\geq0}$ is ultimately periodic (using either~\eqref{eq:Aaffine} or~\eqref{eq:HJ}). However, it seems not easy to deduce~\eqref{eq:Aaffine} an expression of the minimal period. This can be done through a third expression for $\aff{A}_{n-1}^{FC}(q)$ that we prove combinatorially now.

\begin{proposition}\label{prop:aaffine}
For any positive $n$, the generating function $\aff{A}_{n-1}^{FC}(q)$ satisfies
\begin{equation}
 \label{eq:newAaffine}
\aff{A}_{n-1}^{FC}(q)\equiv \frac{1}{1-q^{n}}\left(\qbi{2n}{n}{q}-2\right).
\end{equation}
\end{proposition}

\begin{proof}

We have to count FC heaps, corresponding to FC elements in $\Aaff_{n-1}$, with respect to their number of vertices. Take a large enough such FC heap $H$: we will need to assume that $|H_{s_0}|> n/2$, which holds as soon as $|H|$ is large enough (in fact $|H|\geq n^2$  as is easily seen through the alternating condition). For any $i\in\{0,\ldots,n-1\}$, denote the elements of the chain $H_{s_i}$ by $s_i^{(1)}<s_i^{(2)}<\cdots<s_i^{(h_i)}$. Let $k$ be the number of such indices $i$ satisfying $s_{i+1}^{(1)}<s_{i}^{(1)}$. Notice that $k\in\{1,\ldots,n-1\}$. 

Consider now the ascending chain $s_0^{(1)}<s_1^{(j_1)}<\cdots<s_{n-k}^{(j_{n-k})}$ where $j_0:=1$, $j_{i+1}-j_{i}$ is $1$ if $s_{i+1}^{(1)}<s_{i}^{(1)}$ and $0$ otherwise. Consider also the descending chain $s_{n-k}^{(j_{n-k})}>s_{n-k+1}^{(j_{n-k+1})}>\cdots>s_n^{(j_n)}=s_0^{(j_n)}$ where $j_{i+1}-j_{i}$ is $-1$ if $s_{i+1}^{(1)}>s_{i}^{(1)}$ and $0$ otherwise. Now $j_n=1$ as a quick computation will show; the construction is illustrated in Figure~\ref{fig:decompoAtilde}, left.

 Call $H_{low}$ the poset induced by the $s_i^{(j)}$ for all $i\in\{0,\ldots,n-1\}$ and all $j\in\{1,\dots,j_i\}$. It is isomorphic to a Ferrers diagram included in the box $k\times(n-k)$, where $s_i$-vertices correspond to cells in the $i$th diagonal of the box; conversely, all such diagrams give valid posets $H_{low}$, and we recall that they have $\qbi{n}{k}{q}$ as generating polynomial by Lemma~\ref{lemma:partitionsboite}. The same construction can be made on the top part of the heap (which amounts to performing our construction on the dual heap of $H$), resulting in a poset $H_{high}$. Remark that thanks to our assumption $|H_{s_0}|> n/2$, it is easily seen that $H_{low}$ and $H_{high}$ are disjoint. Notice that the same integer $k$ occurs in both constructions of $H_{low}$ and $H_{high}$.

The remaining vertices of $H$ are easy to count; write $m=|H_{s_0}|$. Then all the $s_i$-vertices in $H-(H_{low}\cup H_{high})$, for $i=0,1,\ldots,n-1$, are counted by 
\[m,m-1,\ldots,\underbrace{m-k,m-k,\ldots,m-k}_{n-2k-1},m-k+1,\ldots,m-1,\]
so  the total size of $H-(H_{low}\cup H_{high})$ is given by $(m-k)n+k^2$. 

\begin{figure}[!ht]
\begin{center}
\includegraphics[height=8cm]{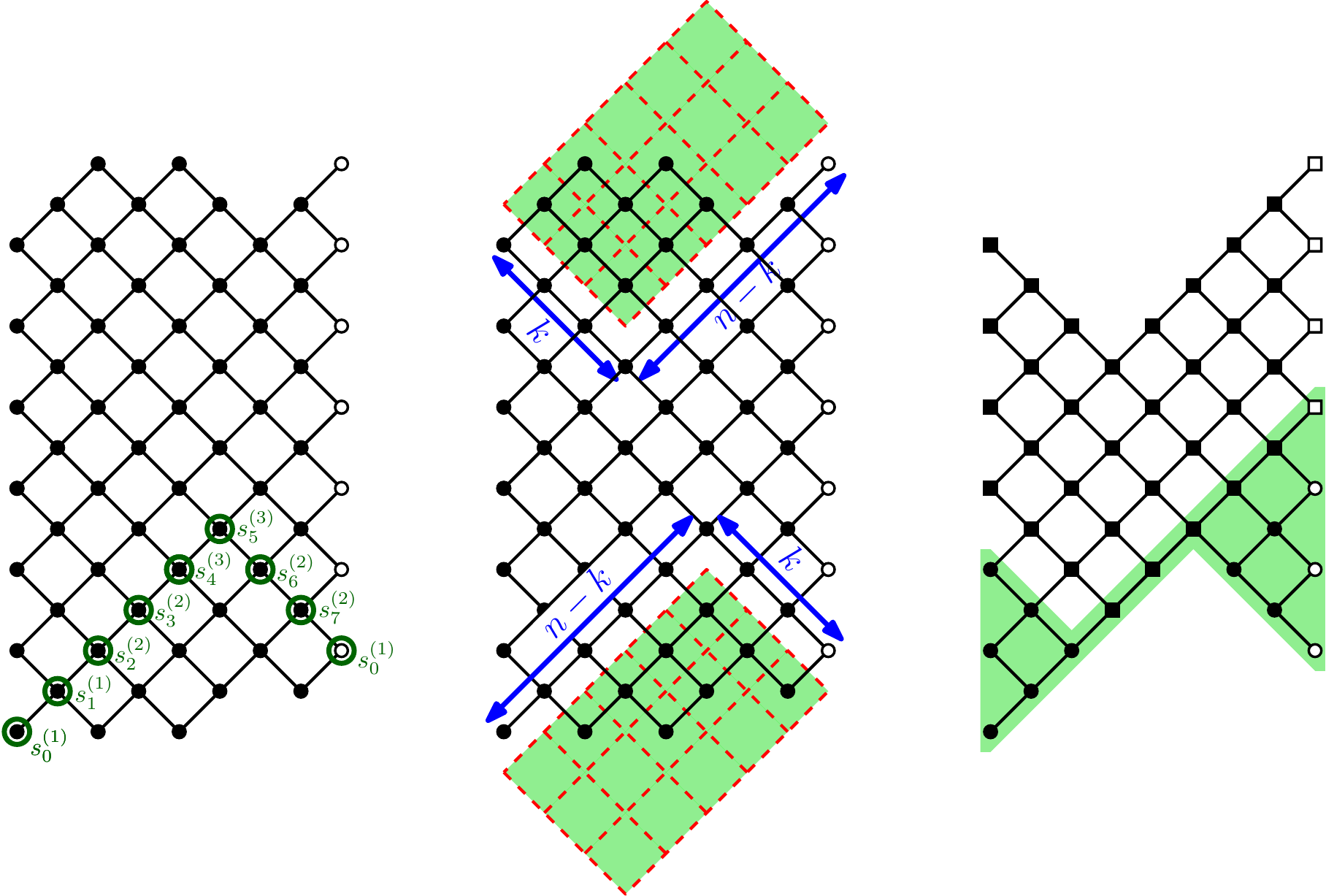}
\caption{\label{fig:decompoAtilde} Decomposition of a FC heap of type $\Aaff_7$.}
\end{center}
\end{figure}

Therefore we have
\begin{align*}
\aff{A}_{n-1}^{FC}(q)&\equiv\sum_{m>n/2}\sum_{k=1}^{n-1}\qbi{n}{k}{q}^2q^{(m-k)n+k^2}=\frac{q^{\lceil n/2 \rceil n}}{1-q^n}\sum_{k=1}^{n-1}\qbi{n}{k}{q}^2q^{k^2-kn}\\
       &\equiv\frac{1}{1-q^n}\sum_{k=1}^{n-1}\qbi{n}{k}{q}^2q^{k^2},
\end{align*}
which, by using \eqref{durfee} in the case $a=b=n$, is equal to the right-hand side of \eqref{eq:newAaffine}.
\end{proof}

From this formula we can deduce the minimal period in type $\aff{A}_{n-1}$.

\begin{proposition}\label{prop:exactperiodaaffine}
In type $\aff{A}_{n-1}$, the growth function $(a^{(n)}_l)_{l\geq 0}$ is ultimately periodic, with minimal period equal to  $p^{\alpha-1}$ if $n=p^\alpha$ is a prime power, and to $n$ otherwise.
\end{proposition}

\begin{proof}
We expand \eqref{eq:newAaffine} into partial fractions, yielding:
\begin{equation}\label{fractionexpansiona}
\aff{A}_{n-1}^{FC}(q)\equiv\frac{1}{n}\sum_{j=0}^{n-1}\frac{a_{n,j}}{1-q\xi_n^{j}},
\end{equation}
where $\displaystyle a_{n,j}=\qbi{2n}{n}{\xi_n^{n-j}}-2$ thanks to  \eqref{eq:parfrac}. From Lemma~\ref{lemma:qbinrootunity} with $n$ replaced by $2n$, $k$ by $n$ and $j$ by $2j$, we get by setting $d:=(n,j)$:
\begin{equation}\label{coeffa}
a_{n,j}=\bi{2(n,n-j)}{(n,n-j)}-2=\bi{2d}{d}-2.
\end{equation}
This shows that $a_{n,j}\neq 0$ if and only if $d> 1 $. We now use Corollary~\ref{coro:partialfractionperiod}, which says that the minimal period is the least common multiple of the numbers $n/d$ for $j=0,1,\ldots,n-1$ such that $d> 1$. Hence this minimal period is also the least common multiple of all {\em strict} divisors of $n$.

If $n=p^\alpha$, these divisors are $p^\beta$ for $\beta=0,\ldots,\alpha-1$, and the least common multiple of these is $p^{\alpha-1}$. If $n$ has more than one prime factor, it can be written $n=n_1n_2$ with $(n_1,n_2)=1$ and $n_1,n_2<n$. Then the least common multiple of $n_1$ and $n_2$ is $n$, which achieves the proof.
\end{proof}

\begin{remark}\label{rk:hanusajones}
Equations \eqref{eq:HJ}--\eqref{eq:newAaffine} give three expressions for the series $\aff{A}_{n-1}^{FC}(q)$ up to a polynomial. This entails that the numerators are equal up to a multiple of $1-q^n$. It seems a challenging problem to prove combinatorially these equalities. In particular, expressions at the numerators on the right sides of  \eqref{eq:HJ} and \eqref{eq:newAaffine} can be identified by noting that 
$\sum_{k=0}^{n}\qbi{n}{k}{q}^2$ and $\qbi{2n}{n}{q}$ are equal for all $q=\xi_n^j$ ($0\leq j\leq n-1$), which is easily seen by using Lemma~\ref{lemma:qbinrootunity}. 
\end{remark}

\begin{corollary}\label{coro:exactperiodinvaaffine}
The growth function $(\overline{a}_l^{(n)})_{l\geq0}$  for fully commutative involutions of type $\Aaff_{n-1}$ is ultimately periodic with minimal period $n$ if $n$ is even. If $n$ is odd, there are finitely many fully commutative involutions of type $\Aaff_{n-1}$.
\end{corollary}
\begin{proof}
As noticed in \cite{BJN_families, St3}, the heaps $H$ corresponding to FC involutions are those that are vertically symmetric. In the proof of Proposition~\ref{prop:aaffine} which focuses on elements of large length, this symmetry condition means that $H_{low}$ and $H_{high}$ are mirror images, which entails in particular $k=n-k$. This shows that there is no such configuration when $n$ is odd, so there are finitely many fully commutative involutions in this case. If $n$ is even, we get easily:
\begin{equation}\label{invaaffinefg}
\aff{A}_{n-1}^{FCI}(q)\equiv\sum_{m>n/2}\qbi{n}{n/2}{q^2}q^{(m-n/2)n+n^2/4}\equiv\frac{q^{n^2/4}}{1-q^n}\qbi{n}{n/2}{q^2},
\end{equation}
and the coefficients appearing in the partial fraction expansion of this series are never equal to $0$. Corollary~\ref{coro:partialfractionperiod} then yields the result.
\end{proof}

\section[Type CBD affine]{Types $\aff{C}$, $\aff{B}$ and $\aff{D}$}
\label{CBDaffine}

As shown in~\cite{BJN}, one has ultimate periodicity also in the coefficients of $\aff{C}^{FC}_n(q)$.  Indeed, from Expression \eqref{eq:Caffine} one obtains that these coefficients have period $n+1$. To obtain an expression of the minimal period, we will use a new expression for $\aff{C}^{FC}_n(q)$.

\begin{proposition}\label{prop:caffine}
For any positive integer $n$, the generating function, according to the length, for fully commutative elements in $\aff{C}_{n}$ takes the form:
\begin{equation}
\label{eq:newCaffine}
\aff{C}_{n}^{FC}(q)\equiv\frac{(-q;q)_n^2}{1-q^{n+1}}+\frac{2n}{1-q}.
\end{equation}
\end{proposition}

\begin{proof}
As explained in Section~\ref{sec:heapsfullycomm}, there are two kinds of long FC elements of type $\aff{C}_{n}$. For the ones corresponding to zigzag heaps, the generating function is given by the second terms in~\eqref{eq:newCaffine}. Therefore it is enough to focus on the alternating elements.

 Fix an alternating  FC element $w\in\aff{C}_{n}$, and denote by $H$ the corresponding alternating heap; we have to count such heaps with respect to their number of vertices. We will need to assume that $|H_{t}|> n$, which holds as soon as $|H|$ is large enough (larger than $3n(n+1)/2$, as can easily be seen from the alternating condition). Set $s_0:=t$ and $s_n:=u$.  For any $i\in\{0,\ldots,n\}$, denote the elements of the chain $H_{s_i}$ by $s_i^{(1)}<s_i^{(2)}<\cdots<s_i^{(h_i)}$. Let $j$ be the number of such indices $i$ satisfying $s_{i+1}^{(1)}<s_{i}^{(1)}$. Notice that $j\in\{0,\ldots,n\}$.

Consider now the ascending chain $s_0^{(1)}<s_1^{(v_1)}<\cdots<s_{n-j}^{(v_{n-j})}$ where $v_0:=1$, $v_{i+1}-v_{i}$ is $1$ if $s_{i+1}<s_{i}$ and $0$ otherwise. Consider also the descending chain $s_{n-j}^{(v_{n-j})}>s_{n-j+1}^{(v_{n-j+1})}>\cdots>s_n^{(v_n)}$ where $v_{i+1}-v_{i}$ is $-1$ if $s_{i+1}>s_{i}$ and $0$ otherwise. Now $v_n=1$ as a quick computation will show. Call $H_{low}$ the subheap with vertices $s_i^{v}$ for all $i\in\{0,\ldots,n\}$ and all $v\in\{1,\dots,v_i\}$: it forms a Ferrers diagram included in the box $j\times(n-j)$,  and any diagram gives a valid  $H_{low}$. By Lemma~\ref{lemma:partitionsboite}, such Ferrers diagrams have $\qbi{n}{j}{q}$ as generating function. The same construction can be made on the top part of the heap (which amounts to performing our construction on heap dual to $H$), resulting in a subset $H_{high}$, corresponding this time to a  Ferrers diagram included in the box $(n-k)\times k$, for a $k\in\{0,\ldots,n\}$.  These constructions are illustrated in Figure~\ref{fig:decompoCtilde}.

\begin{figure}[!ht]
\begin{center}
\includegraphics[height=8cm]{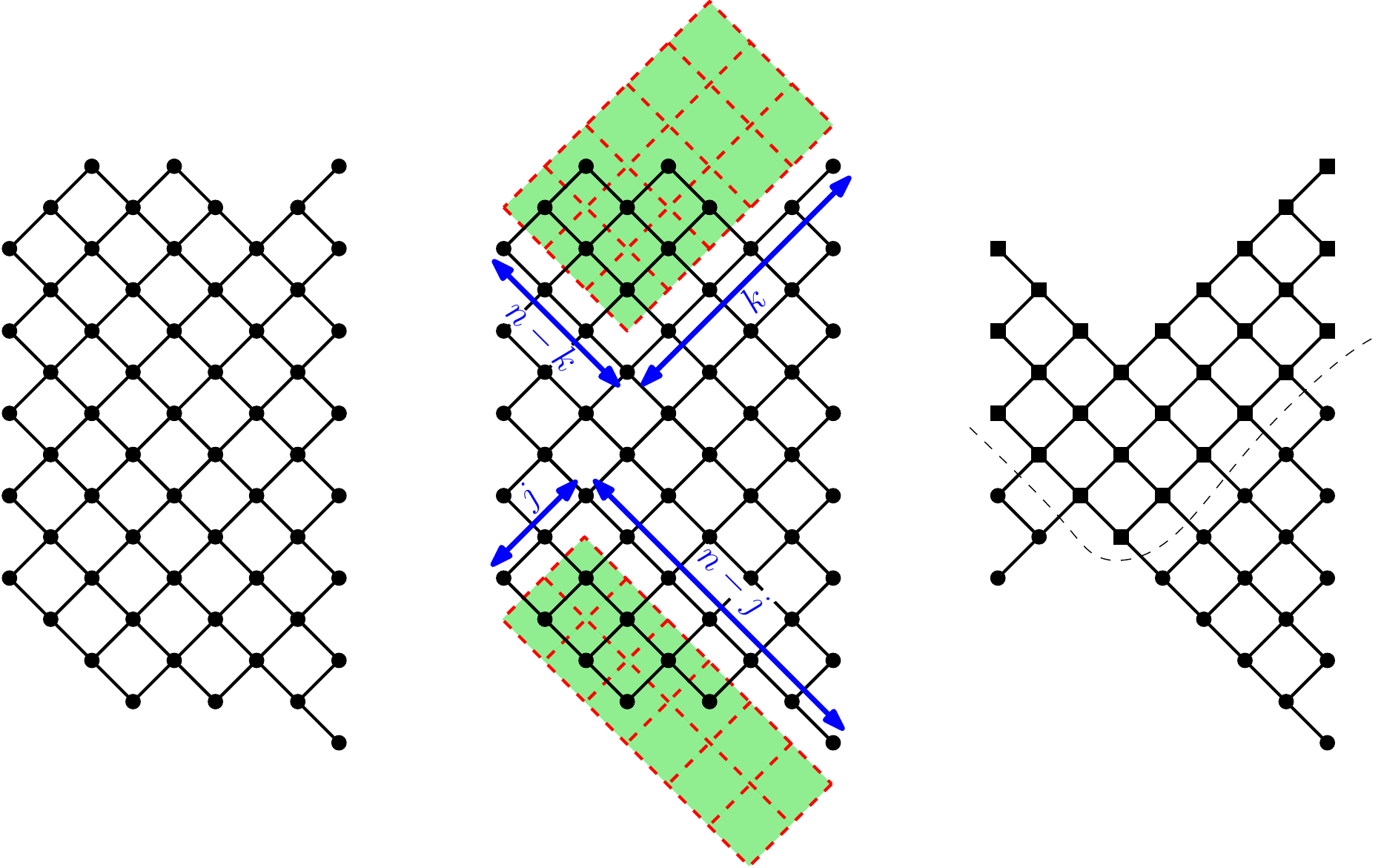}
\caption{\label{fig:decompoCtilde} Decomposition of an alternating FC heap of type $\aff{C}_8$.}
\end{center}
\end{figure}
The remaining vertices of $H$ are easy to count: let us assume  without loss of generality that $n-j\leq k$. Write $m=|H_{s_0}|$. Then the number of $s_i$-vertices in $H-(H_{low}\cup H_{high})$, for $i=0,1,\ldots,n$, is given by 
\[m,m-1,\ldots,\underbrace{m-(n-j),\ldots,m-(n-j)}_{k+j-n+1},m-(n-j)+1,\ldots,m+j-k,\]
so  the total size of $H-(H_{low}\cup H_{high})$ is given by $(m-(n-j))(n+1)+\bi{n-j+1}{2}+\bi{n-k+1}{2}$. 
Therefore, replacing  $j$ (\emph{resp.} $k$)  by $n-j$ (\emph{resp.} $n-k$), we have
\begin{eqnarray*}
\aff{C}_{n}^{FC}(q)-\frac{2n}{1-q}&\equiv&\sum_{m>n}\sum_{k=0}^{n}\sum_{j=0}^n\qbi{n}{k}{q}\qbi{n}{j}{q}q^{\bi{j+1}{2}+\bi{k+1}{2}}\\
&=&\frac{q^{n+1}}{1-q^{n+1}}\left(\sum_{i=0}^{n}\qbi{n}{i}{q}q^{i(i+1)/2}\right)^2,
\end{eqnarray*}
which, by using \eqref{staircase}, yields \eqref{eq:newCaffine}.

\end{proof}

We can deduce the minimal period in type $\aff{C}_{n}$. Let $(c^{(n)}_l)_{l\geq 0}$ be the growth function of FC elements in type $\aff{C}_{n}$. In the rest of this section we let $2m+1$ be the largest odd divisor of $n+1$, and we write $n+1=2^{\alpha}(2m+1)$.

\begin{proposition}\label{prop:exactperiodcaffine}
 The growth function $(c^{(n)}_l)_{l\geq 0}$ is ultimately periodic with minimal period equal to $2m+1$.
\end{proposition}
\begin{proof}
We expand \eqref{eq:newCaffine} into partial fractions, yielding:
\begin{equation}\label{eq:exptypectilde}
\aff{C}_{n}^{FC}(q)\equiv\frac{1}{n+1}\sum_{j=0}^{n}\frac{c_{n,j}}{1-q\xi_{n+1}^{j}}+\frac{2n}{1-q},
\end{equation}
where by~\eqref{eq:parfrac} we have
\begin{equation}
\label{eq:cnj_long}
c_{n,j}=\left.(-q;q)_n^2\right|_{q=\xi_{n+1}^{-j}}=\left(\prod_{l=1}^{n}(1+\xi_{n+1}^{-jl}\right)^2=\frac{1}{4}\left(\prod_{l=0}^{n}(1+\xi_{n+1}^{-jl}\right)^2.
\end{equation}
Now $\xi_{n+1}^{-j}$ has order $f:=(n+1)/d$, where $d:=(n+1,j)$, so we get  
$$c_{n,j}=\frac{1}{4}\left(\prod_{l=0}^{f-1}(1+\xi_{f}^l)\right)^{2d}.$$ Moreover $\displaystyle X^{f}-1=\prod_{l=0}^{f-1}(X-\xi_{f}^l)$ yields $\displaystyle(-1)^{f}-1=(-1)^{f}\prod_{l=0}^{f-1}(1+\xi_{f}^l)$ so that
\begin{equation}
\label{eq:coeffc}
c_{n,j}=\frac{1}{4}\left(1-(-1)^{\frac{n+1}{d}}\right)^{2d}.
\end{equation}
This shows that $c_{n,j}\neq 0$ if, and only if $(n+1)/d$ is odd (this can also be deduced directly from~\eqref{eq:cnj_long}). From Corollary~\ref{coro:partialfractionperiod}, we know that the minimal period of $\aff{C}_{n}^{FC}(q)-2n/(1-q)$ is equal to the least common multiple of all odd divisors of $n+1$, which completes the proof, as $2n/(1-q)$ has period $1$.
\end{proof}

\begin{corollary}\label{coro:exactperiodinvcaffine}
The growth function $(\overline{c}_l^{(n)})_{l\geq0})$ of affine fully commutative involutions in type $\aff{C}_{n}$ is ultimately periodic with minimal period $2(2m+1)$.
\end{corollary}
\begin{proof}
To obtain an expression for FC involutions, we need to consider those heaps described in the proof of Proposition~\ref{prop:caffine} which are vertically symmetric. From~\cite{BJN_families}, we know that the corresponding zigzag heaps have $2q^{2n+3}/(1-q^2)$ as generating function, so only focus on alternating heaps. Here the vertical symmetry means that $j=n-k$ and the two Ferrers diagrams identified in the proof of Proposition~\ref{prop:caffine} have to be identical. This shows that
\begin{eqnarray*}
\aff{C}_{n}^{FCI}(q)-\frac{2q^{2n+3}}{1-q^2}&\equiv&\frac{1}{1-q^{n+1}}\sum_{k=0}^{n}\qbi{n}{k}{q^2}q^{(n-k+1)(n-k+2)/2+(k+1)(k+2)/2}\\
&=&\frac{1}{1-q^{n+1}}\sum_{k=0}^{n}\qbi{n}{k}{q^2}q^{k^2-nk+2+n(n+3)/2}\\
&\equiv& \frac{q^{1+(n+1)(n+2)/2}}{1-q^{n+1}}\sum_{k=0}^{n}\qbi{n}{k}{q^2}q^{k^2+k}\\
&=&\frac{q^{1+(n+1)(n+2)/2}}{1-q^{n+1}}(-q^2;q^2)_n,
\end{eqnarray*}
where we have used~\eqref{staircase} with $q$ replaced by $q^2$ for proving the last equality. Finally, the partial fraction decomposition of $\aff{C}_{n}^{FCI}(q)$ takes the form
$$\aff{C}_{n}^{FCI}(q)\equiv\frac{1}{n+1}\sum_{j=0}^{n}\frac{\overline{c}_{n,j}}{1-q\xi_{n+1}^{j}}+\frac{2q}{1-q^2},$$
where by~\eqref{eq:parfrac}
$$\overline{c}_{n,j}=\left.q^{1+(n+1)(n+2)/2}(-q^2;q^2)_n\right|_{q=\xi_{n+1}^{-j}}=\frac{(-1)^{nj}\xi_{n+1}^{j}}{2}\prod_{l=0}^{n}(1+\xi_{n+1}^{2jl}).$$
Set as usual $d:=(n+1,j)$ and $f:=(n+1)/d$, yielding
$$\overline{c}_{n,j}=\frac{(-1)^{nj}\xi_{n+1}^{j}}{2}\left(\prod_{l=0}^{f-1}(1+\xi_{f}^{2l})\right)^{d}.$$
From the factorization of $\displaystyle X^{f}-1$, we derive 
$$\sqrt{-1}^f-1=\prod_{l=0}^{f-1}(\sqrt{-1}-\xi_f^l)\quad\mbox{and}\quad (-\sqrt{-1})^f-1=\prod_{l=0}^{f-1}(-\sqrt{-1}-\xi_f^l),$$ 
and plugging these into the previous expression for $\overline{c}_{n,j}$ yields
\begin{equation}\label{coeffcinv}
\overline{c}_{n,j}=\frac{(-1)^{nj}\xi_{n+1}^{j}}{2}\left(\left(1-\sqrt{-1}^{\frac{n+1}{d}}\right)\left(1-(-\sqrt{-1})^{\frac{n+1}{d}}\right)\right)^{d}.
\end{equation}
This shows  that $\overline{c}_{n,j}\neq 0$ if and only if $(n+1)/d$ is not equal to $0$ modulo $4$. From Corollary~\ref{coro:partialfractionperiod}, we know that the minimal period of $\aff{C}_{n}^{FCI}(q)-2q/(1-q^2)$ is equal to the least common multiple of all  divisors of $n+1$ which are not congruent to $0$ modulo $4$. This is either $2m+1$ or $2(2m+1)$, and as the period corresponding to the term $2q/(1-q^2)$ is $2$, the conclusion follows.
\end{proof}

Finally, we summarize the results corresponding to the affine types  $\aff{B}_{n+1}$ and $\aff{D}_{n+2}$; this achieves the proof of the main theorem stated in the introduction.

\begin{corollary}\label{coro:exactperiodbdaffine}
With the same notations as in Proposition~\ref{prop:exactperiodcaffine}, the growth function in type $\aff{B}_{n+1}$ (\emph{resp.} $\aff{D}_{n+2}$) is ultimately periodic, with minimal period equal to $(2n+1)(2m+1)$  (\emph{resp.} $n+1$).
\end{corollary}
\begin{proof}
In~\cite{BJN}, the following identity was proved:
\begin{equation}
\label{eq:Btilde}
 \aff{B}^{FC}_{n+1}(q)\equiv \frac{2q^{n+1}\touch{\Gen}_n(q)}{1-q^{n+1}}+\frac{(2n+3)q^{2n+4}}{1-q}+\frac{q^{2(2n+1)}}{1-q^{2n+1}}.
\end{equation}
Thanks to the proof of Proposition~\ref{prop:exactperiodcaffine}, the minimal period corresponding to the first term in \eqref{eq:Btilde} is $2m+1$. Moreover, the period corresponding to the second term is trivially equal to $1$, and the conclusion follows by noting that the minimal period corresponding to the third term is $2n+1$, which is relatively prime to $2m+1$.

Moreover, we have the following expression proved  in~\cite{BJN}:
\begin{equation}
\label{eq:Dtilde}
 \aff{D}^{FC}_{n+2}(q)\equiv \frac{4q^{n+1}\touch{\Gen}_n(q)}{1-q^{n+1}}+\frac{(2n+6)q^{2n+5}}{1-q}+\frac{2q^{3(n+1)}}{1-q^{n+1}}.
\end{equation}
The minimal periods corresponding to each term are $2m+1,1$ and $n+1$ respectively. Since $2m+1$ divides $n+1$, the conclusion follows.
\end{proof}

  The case of involutions in these types is easily derived by the same methods, and is left to the interested reader.

\section{Exact and asymptotic evaluations}
\label{section:num}

In this section, we use our results to give some explanations to numerical observations regarding the repartition of the number of FC elements on a period. Let us start with the example of type $\Aaff_{n-1}$ with $n=10$. Thanks to Proposition~\ref{prop:exactperiodaaffine}, we know that the minimal (ultimate) period of the growth function $\left(a_l^{(10)}\right)_{\ell\geq 0}$ is equal to $10$.

 The numbers $a_l^{(10)}$ for $l=1,2,\ldots ,10$ modulo 10 are given by
\[
18450, 18500, 18450, 18500, 18452, 18500, 18450, 18500, 18450, 18502
\]
in the periodic range. There are very small variations between these values, which will be explained by the results of this section.
\smallskip
 
Given two integers $r>0,l\geq0$, the {\em Ramanujan sum} $Ram_{r}(l)$ (see~\cite{ArithmeticalFunctions,TothNotes}) is defined as the sum of $l$th powers of the primitive $r$th roots of unity
\[\displaystyle Ram_{r}(l):=\sum_{1\leq j\leq r\atop (j,r)=1}\xi_r^{lj}.\]

\begin{proposition}
\label{prop:ramanujanaaffine}
In type $\Aaff_{n-1}$, the growth function satisfies for  any large enough integer $l$: 
\begin{equation}\label{ramansum}
a_l^{(n)}=\frac{1}{n}\sum_{d|n\,,\,d>1}\left(\bi{2d}{d}-2\right)Ram_{\frac{n}{d}}(l).
\end{equation}
\end{proposition}
\begin{proof}
From~\eqref{coeffa}, we see that the coefficient $a_{n,j}$ in \ref{fractionexpansiona} only depends on $d:=(n,j)$. Therefore by taking the coefficients of $q^l$  on both sides of\eqref{fractionexpansiona}, we obtain:
\begin{eqnarray*}
a_l^{(n)}&=&\frac{1}{n}\sum_{j=1}^{n}a_{n,j}\xi_n^{-lj}\\
&=&\frac{1}{n}\sum_{d|n}\left(\bi{2d}{d}-2\right)\sum_{1\leq j'\leq n/d\atop (j',n/d)=1}(\xi_{n/d}^{-j'})^l,
\end{eqnarray*}
where we set $j'=j/d$ and wrote $\xi_n^{-lj}=\xi_{n/d}^{-lj'}$. This is the desired expression (note that the term for $d=1$ vanishes).
\end{proof}

Now all Ramanujan sums are obviously bounded by $n$. Moreover, the dominant term in the sum~\eqref{ramansum} is given by $d=n$, for which the Ramanujan sum is constant equal to $1$.  As the next dominant term is given by $d=\lfloor n/2\rfloor$, we can write by using Stirling's asymptotic formula for the factorials:
$$a_l^{(n)}=\frac{\bi{2n}{n}}{n}(1+\mbox{O}(n\,2^{-n})),\;\;n\to+\infty.$$
We deduce that for $n$ and $l$ large enough, the number $a_l^{(n)}$ is closed to the mean value of the growth function $\left(\bi{2n}{n}-2\right)/n$, which is $18475.4$ when $n=10$.

\medskip
We have the following analogous result in type $\aff{C}_{n}$.
\begin{proposition}\label{prop:ramanujancaffine}
Write $n+1=2^\alpha(2m+1)$. For any large enough integer $l$, the growth function in type $\aff{C}_{n}$ satisfies:
\begin{equation}
\label{ramansumc}
c_l^{(n)}=2n+\frac{1}{4(n+1)}\sum_{u|2m+1}2^{u2^{\alpha+1}}Ram_{\frac{2m+1}{u}}(l).
\end{equation}
\end{proposition}

\begin{proof}
By proceeding as in the proof of Proposition~\ref{prop:ramanujanaaffine} and using this time~\eqref{eq:coeffc} and~\eqref{eq:exptypectilde}, we get:
\[
c_l^{(n)}=2n+\frac{1}{4(n+1)}\sum_{d|n+1}\left(1-(-1)^{\frac{n+1}{d}}\right)^{2d}Ram_{\frac{n+1}{d}}(l).
\]
To finish, notice that a term in the sum is zero unless $d=2^\alpha u$ for $u$ a divisor of $2m+1$.  
\end{proof}

This shows that again, for large $n$, in the periodic range, the number $c_l^{(n)}$ is close to the mean value $2n+4^n/(n+1)$.
\medskip

It is possible to write the same kinds of expressions regarding the number of FC involutions of length $\ell$ in all affine types. We will not give details here, and we just mention that the same conclusions occur regarding the repartitions of the values on a period. 

We shall simply indicate a striking observation which concerns the case of type $\Aaff_{n-1}$: there exists a simple relation between growth functions for FC elements and FC involutions (for long elements). For all $m$, and large enough $l$, there holds
$$\overline{a}_l^{(2m)}=\begin{cases}
 a_{l/2}^{(m)}+2\chi(m|l/2)&\mbox{if } l\mbox{ and } m\mbox{ are even}, \\
 a_{(l+m)/2}^{(m)}+2\chi(m|(l+m)/2) &\mbox{if } l\mbox{ and } m\mbox{ are odd},\\ 
0 & \mbox{if } l\mbox{ and } m\mbox{ have opposite parity}. \end{cases}$$

Here $\chi(Y)$ is the \emph{true-false} function on the property $Y$, which is equal to $1$ if $Y$ is true, and $0$ otherwise. To prove these relations, replace $q$ by $q^2$ in~\eqref{eq:newAaffine} and take $n=2m$ in~\eqref{invaaffinefg}. This yields
$$\aff{A}_{2m-1}^{FCI}(q)\equiv q^{m^2}\left(\aff{A}_{n-1}^{FC}(q^2)+\frac{2}{1-q^{2m}}\right),$$
and we can conclude by noting that $m^2$ is congruent to $0$ (\emph{resp.} $m$) modulo $2m$ if $m$ is even (\emph{resp.} odd).

\section{A cyclic sieving phenomenon}
\label{sec:csp}

Let $X$ be a finite set endowed with the action of a finite cyclic group $C=\left<c\right>$ of order $n$. Let also $P$ be a polynomial in $\mathbb{N}[q]$. Denote by $X^g$ the subset of elements of $X$ fixed by $g\in C$, and recall that $\xi_n:=\mbox{e}^{2i\pi/n}$. Then the triple $(X,C,P)$ exhibits the {\em cyclic sieving phenomenon} (see ~\cite{reiner2004cyclic, sagan2010cyclic}) if 
\begin{equation}
\label{eq:csp}
P(\xi_n^j)=|X^{c^j}|\quad\text{for any }j\in\{0,\ldots,n-1\}.
\end{equation}

Here we take for $X$ the set $\touch{\Cylset}_n$ of  lattice paths defined in Section~\ref{sec:heapsfullycomm}, where the cyclic action is generated by the rotation $\mathbf{r}$ which rotates paths one unit to the right. Finally we choose as polynomial $\touch{\Cyl}_n(q)$ which enumerates the paths in $\touch{\Cylset}_n$  according to their area.

\begin{proposition}
The triple $(\touch{\Cylset}_n,\left<\mathbf{r}\right>,\touch{\Cyl}_n(q))$ exhibits the cyclic sieving phenomenon.
\end{proposition}

\begin{proof}
We will evaluate both sides of~\eqref{eq:csp}. For the r.h.s., we need to count paths fixed by a power $\mathbf{r}^j$. First notice that $\mathbf{r}^j$ and $\mathbf{r}^{(n,j)}$ generate the same subgroup of $C$ (they have the same order $n/(n,j)$), hence it is equivalent to count paths fixed by $\mathbf{r}^{(n,j)}$. Such paths are clearly concatenations of $n/(n,j)$ identical paths of length $n$, where  the repeated portion is allowed to be any element of $\touch{\Cylset}_{(n,j)}$. Since this last set has cardinality $\binom{2(n,j)}{(n,j)}$, we obtain 

\[\left|\touch{\Cylset}_n^{\mathbf{r}^j}\right|=\binom{2(n,j)}{(n,j)}.\]

Now we need to evaluate the polynomial $\touch{\Cyl}_n(q)$ at $q=\xi_n^j$. Note that this is \emph{a priori} not obvious, since the polynomial does not possess a nice expression as far as we know: it can only be computed recursively thanks to certain functional equations from~\cite{BJN}.

  The idea is to use the observation from Remark~\ref{rk:hanusajones}: we have for a certain polynomial $Q(q)$ the equality \[\touch{\Cyl}_n(q)=\qbi{2n}{n}{q} + (1-q^n)Q(q),\]
so both polynomials $\touch{\Cyl}_n(q)$ and $\qbi{2n}{n}{q}$ take the same values at $n$th roots of unity.  These were calculated in Lemma~\ref{lemma:qbinrootunity}, and match indeed the values found for the fixed points.
\end{proof}

\bibliographystyle{plain}
\bibliography{fullycommut}

\end{document}